\documentclass[11pt]{article}
\usepackage{amsthm}
\usepackage{tikz-cd}
\usepackage{shuffle}
\usepackage{times}
\usepackage{amsmath}
\usepackage{amssymb}
\usepackage{hyperref}
\usepackage[nottoc,numbib]{tocbibind}
\newtheorem{thm}{Theorem}[section]
\newtheorem{rmk}[thm]{Remark}

\newtheorem{lem}[thm]{Lemma}

\newtheorem{prop}[thm]{Proposition}
\newtheorem{cor}[thm]{Corollary}

\begin{document}
\title{A finiteness result on representations of the profinite étale Nori fundamental group}
\author{Xiaodong Yi\footnote{yixd97@outlook.com}}
\date{}
\maketitle
\begin{abstract}
Let $(X,x)$ be a pointed geometrically connected smooth projective variety over a sub-$p$-adic field $k$. For any given rank $n$, we prove that there are only finitely many isomorphism classes of representations $\pi_{1}^{N}(X,x)\rightarrow \mathrm{GL}_{n}$, where $\pi_{1}^{N}(X,x)$ is the Nori fundamental group of essentially finite bundles. Equivalently, there are only finitely many isomorphism classes of essentially finite bundles of rank $n$. This answers a question from Gasbarri. We also discuss analogues of this finiteness result in characteristic $p>0$. 
\end{abstract}
\noindent \textbf{Mathematics Subject Classification (MSC 2020):} 14F06,14F35, 14L15,14G20.\par
\section{Introduction}
Let $X$ be a geometrically connected smooth projective variety over a  field $k$. In \cite{nori1982fundamental} Nori introduces the notion of essentially finite bundles on $X$. Provided with a base point $x\in X(k)$, the category of essentially finite bundles is neutral Tannakian, and the resulting Tannakian group scheme $\pi_{1}^{N}(X,x)$ classifies pointed finite torsors above $(X,x)$. By Tannakian duality there is an equivalence between the category of finite dimensional representations of $\pi_{1}^{N}(X,x)$ and the category of essentially finite bundles.  \par
Now we further assume $k$ to be a number field and $X$ to be a curve. Here is a simple observation: essentially finite line bundles are torsion, so their isomorphism classes are in bijection with torsion elements of $\mathrm{Pic}_{X}^{0}(k)$, which are finite by Mordell-Weil theorem. In \cite{10.1215/S0012-7094-03-11723-8} Gasbarri asks for a non-abelian analogue:  are there only finitely many isomorphism classes of representations $\pi_{1}^{N}(X,x)\rightarrow \mathrm{GL}_{n}$ of given rank $n$? Equivalently, are there only finitely many isomorphism classes of essentially finite bundles of given rank $n$? Gasbarri provides two evidences towards an affirmative answer: if $X$ has good reduction at a finite place $\mathfrak{p}$ over a prime $p$, there are only finitely many isomorphism classes of representations $\pi_{1}^{N}(X,x)^{p'}\rightarrow \mathrm{GL}_{n}$, where $\pi_{1}^{N}(X,x)^{p'}$ is the maximal prime-to-$p$ quotient of $\pi_{1}^{N}(X,x)$. The proof is via specialization of étale fundamental groups as in \cite{grothendieck2003revetements} and the boundedness of essentially finite bundles as they are semistable of degree $0$. See Theorem 3.3 \cite{10.1215/S0012-7094-03-11723-8}; let $(\mathcal{X},x)$ be a pointed arithmetic surface over the ring of integers $\mathfrak{o}_{k}$ of $k$. Gasbarri defines a Nori fundamental group scheme $\pi_{1}^{N}(\mathcal{X},x)$ over $\mathfrak{o}_{k}$, and he proves that  there are only finitely many isomorphism classes of representations $\pi_{1}^{N}(\mathcal{X},x)_{k}\rightarrow\mathrm{GL}_{n}$ from the generic fiber. The proof is via Arakelov geometry. See Theorem 4.1 \cite{10.1215/S0012-7094-03-11723-8}. Note we should not confuse $\pi_{1}^{N}(\mathcal{X},x)_{k}$ with the usual fundamental group scheme $\pi_{1}^{N}(\mathcal{X}_{k},x_{k})$, though their relation has been discussed in \cite{antei2010comparison}. \par
In this article, we will answer Gasbarri's question affirmatively:
\begin{thm}[Corollary \ref{main_cor}]\label{intro_1}
Let $k$ be a sub-$p$-adic field (so typically $k$ can be any number field or $p$-adic field). Let $(X,x)$ be a pointed geometrically connected smooth projective variety over $k$. Then for a given rank $n$ there are only finitely many isomorphism classes of representations $\pi_{1}^{N}(X,x)\rightarrow\mathrm{GL}_{n}$. Equivalently, there are only finitely many isomorphism classes of essentially finite bundles of rank $n$.
\end{thm}
Our argument is different from that from \cite{10.1215/S0012-7094-03-11723-8}, summarized as follows. First, we prove a variant of Jordan's theorem on finite abstract groups embedded into the group of invertible matrices. Our variant is for finite group schemes. Using this variant  we reduce Theorem \ref{intro_1} to the case of essentially finite bundles with abelian monodromy, and then to a finiteness problem on torsion points on abelian varieties. By a standard density argument and a spreading out argument,  we further reduce to the case of $p$-adic fields, which follows from Mattuck's structure theorem on rational points on an abelian variety over a $p$-adic field. \par
In characteristic $p>0$, we establish some analogous finiteness results for the maximal prime-to-$p$ quotient of the Nori fundamental group.
\subsection*{Notation and Convention}
All rings are assumed to be commutative and with a unit. 
\begin{enumerate}
\item For a field $k$, we use $\overline{k}$ to denote a fixed algebraic closure, and $k^{s}$ to denote a fixed separable closure. 
\item For morphisms $X\rightarrow S$ and $S'\rightarrow S$ between schemes, we write $X_{S'}$ for the fiber product $X\times_{S}S'$. For a quasi-coherent sheaf $\mathcal{E}$ on $X$ we write $\mathcal{E}_{S'}$ for its pullback to $X_{S'}$. In case $S'=\mathrm{Spec}\,R'$, we also write $X_{R'}$ for $X_{S'}$ and $\mathcal{E}_{R'}$ for $\mathcal{E}_{S'}$.
\item Let $\kappa(x)$ denote the residue field of a point $x$ on a scheme $X$.
\item  Let $k$ be a field. By a variety over $k$ we mean an integral scheme which is of finite type and separated over $k$.  A pointed variety over $k$ is a pair $(X,x)$, which is a variety $X$ over $k$ provided with $x\in X(k)$. 
\item Let $k$ be a field. We write $\mathrm{GL}_{n,k}$ to be the general linear group scheme over $k$. We also write $\mathrm{GL}_{n}$ when $k$ is clear from the context. By a rank $n$ representation of an affine group scheme $\pi$ over $k$,we mean a morphism $\pi\rightarrow \mathrm{GL}_{n}$ between affine group schemes. Two representations are said to be isomorphic if they differ up to a conjugation by an element from $\mathrm{GL}_{n}(k)$.
\item Let $(X,x)$ be a pointed smooth proper variety over a field $k$. By the monodromy group of an essentially finite bundle $\mathcal{E}$ of rank $n$, we mean the finite group scheme over $k$ which is the image of the representation $\pi_{1}^{N}(X,x)\rightarrow\mathrm{GL}_{n}$ associated to $\mathcal{E}$.
\item Let $A$ be an abelian variety over a field $k$, $k\hookrightarrow l$ be a field extension, and $p$ be a prime number. We write $A(l)_{tor}$ for the group of torsion $l$-points and $A(l)_{tor}^{p'}$ for the group of torsion $l$-points of order coprime to $p$.
\end{enumerate}
\subsection*{Acknowledgement}
The author would like to thank Professor C.Gasbarri for his encouragement to generalize his result in greater generality. The author would like to thank the reviewer for corrections and suggestions. 
\section{The main context}
We refer to  \cite{nori1982fundamental} and Chapter 6 \cite{galfun} for a complete treatment of Nori's fundamental group of essentially finite bundles. Let $(X,x)$ be a pointed geometrically connected smooth proper variety over a field $k$. A vector bundle on $X$ is said to be Nori semistable if its pullback along an arbitrary morphism $C\rightarrow X$ from a smooth proper curve is semistable of degree $0$. Nori semistable bundles form a $k$-linear abelian tensor category. For any monic polynomial $f=x^{n}+a_{n-1}x^{n-1}+...+a_{0}$ with non-negative integral coefficients and any vector bundle $\mathcal{E}$, we write $f(\mathcal{E})$ for the vector bundle \[\mathcal{E}^{\otimes n}\oplus (\mathcal{E}^{\otimes n-1})^{\oplus a_{n-1}}\oplus ... \oplus\mathcal{E}^{\oplus a_{1}}\oplus \mathcal{O}_{X}^{\oplus a_{0}}.\] A vector bundle $\mathcal{E}$ is said to be finite if there exist $f\neq g$ with $f(\mathcal{E})\cong g(\mathcal{E})$. The category of essentially finite bundles is the $k$-linear abelian tensor subcategory of the category of Nori semistable bundles spanned by finite bundles and objects inside it are called essentially finite bundles. Essentially finite line bundles are exactly torsion line bundles. The category of essentially finite bundles is Tannakian, neutralized by the fiber functor defined by $x$, and the corresponding Tannakian group scheme is denoted by $\pi_{1}^{N}(X,x)$, called Nori's fundamental group scheme of essentially finite bundles. Nori's fundamental group $\pi_{1}^{N}(X,x)$ classifies pointed finite torsors in the sense that, giving a pointed torsor $(Y,y)\rightarrow (X,x)$ under a finite group scheme $\mathcal{G}$ over $k$ is equivalent to giving a morphism $\pi_{1}^{N}(X,x)\rightarrow \mathcal{G}$ between group schemes. We include two further properties in characteristic $0$: the category of essentially finite bundles is semisimple (i.e. all short exact sequences split) so essentially finite bundles are finite; the homotopy exact sequence of étale fundamental groups 
\[1\rightarrow\pi_{1}^{et}(X_{\overline{k}},x_{\overline{k}})\rightarrow \pi_{1}^{et}(X,x)\rightarrow\mathrm{Gal}(\overline{k}/k)\rightarrow 1\]
is split by the section $\mathrm{Gal}(\overline{k}/k)\rightarrow \pi_{1}^{et}(X,x)$ defined by $x$. The splitting induces an  action of $\mathrm{Gal}(\overline{k}/k)$ on $\pi_{1}^{et}(X_{\overline{k}},x_{\overline{k}})$, and $\pi_{1}^{et}(X_{\overline{k}},x_{\overline{k}})$ descends to a pro-finite étale group scheme over $k$, which is $\pi_{1}^{N}(X,x)$. 
\begin{lem}[Proposition 3.5 \cite{10.1215/S0012-7094-03-11723-8}]
\label{field_ext}
Let $(X,x)$ be a pointed geometrically connected smooth proper variety over an infinite field $k$, $k\hookrightarrow l$ be a field extension, and $\mathcal{E},\mathcal{F}$ be essentially finite bundles of the same rank. Then we have $\mathcal{E}\cong \mathcal{F}$ if and only if $\mathcal{E}_{l}\cong\mathcal{F}_{l}$. 
\end{lem}
\begin{proof}
The ``only if'' part is trivial. For the ``if'' part, we regard $\mathrm{Hom}_{\mathcal{O}_{X}}(\mathcal{E},\mathcal{F})$ as an affine space over $k$. Fixing a basis for $\mathcal{E}_{x}$ and for $\mathcal{F}_{x}$, we obtain a determinant map 
\[\mathrm{det}:\mathrm{Hom}_{\mathcal{O}_{X}}(\mathcal{E},\mathcal{F}) \rightarrow \mathbb{A}^{1}\]
between varieties over $k$. We claim not all $k$-rational points of $\mathrm{Hom}_{\mathcal{O}_{X}}(\mathcal{E},\mathcal{F})$ are mapped to $0\in \mathbb{A}^{1}(k)$ by $\mathrm{det}$. Indeed, if this is not the case, $\mathrm{det}$ will be the constant map to $0$, since $k$ is infinite and $k$-rational points form a Zariski dense subset of $\mathrm{Hom}_{\mathcal{O}_{X}}(\mathcal{E},\mathcal{F})$. However, by assumption, the base change
\[\mathrm{det}_{l}:\mathrm{Hom}_{\mathcal{O}_{X}}(\mathcal{E},\mathcal{F})_{l}\cong\mathrm{Hom}_{\mathcal{O}_{X_{l}}}(\mathcal{E}_{l},\mathcal{F}_{l}) \rightarrow \mathbb{A}_{l}^{1}\]
is not constantly $0$, as there exists an isomorphism between  $\mathcal{E}_{l}$ and $\mathcal{F}_{l}$. A contradiction. Therefore there is some $\alpha\in \mathrm{Hom}_{\mathcal{O}_{X}}(\mathcal{E},\mathcal{F})$ with $\mathrm{det}(\alpha)\neq 0$ which induces an isomorphism between $\mathcal{E}$ and $\mathcal{F}$.\par
\end{proof} 
\begin{lem}\label{key_lemma}
Let $(X,x)$ be a pointed smooth proper curve over a field $k$ of characteristic $0$. Suppose that for any $n>0$, the set 
\[\{a\in \mathrm{Pic}_{X}^{0}(\overline{k})_{tor}| [\kappa(a):k]\leq n\}\]
is finite. Then for any $n>0$, there are only finitely many isomorphism classes of essentially finite bundles of rank $n$ on $X$ with abelian monodromy.
\end{lem}
\begin{proof}
Let $n>0$ be any integer. For each $\sigma\in \mathrm{Gal}(\overline{k}/k)$, $\sigma$ induces an automorphism $\mathrm{Spec}\,\overline{k}\rightarrow\mathrm{Spec}\,\overline{k}$, and therefore an automorphism $f_{\sigma}$ of $X_{\overline{k}}$ by base change. Let $\mathcal{E}$ be an essentially finite bundle of rank $n$ with abelian monodromy, so its base change $\mathcal{E}_{\overline{k}}$ splits into a direct sum of torsion line bundles:
\[\mathcal{E}_{\overline{k}}\cong\mathcal{L}_{1}^{\oplus n_{1}}\oplus...\oplus\mathcal{L}_{m}^{\oplus n_{m}},n_{1}+...+n_{m}=n\]
with $\mathcal{L}_{i}\ncong\mathcal{L}_{j}$ for $1\leq i\neq j\leq m$.
We have $f_{\sigma}^{*}\mathcal{E}_{\overline{k}}\cong \mathcal{E}_{\overline{k}}$, so by Krull-Schmidt we have an element $\beta(\sigma)$ from the permutation group $S_{m}$ for $\{1,...,m\}$ with $f_{\sigma}^{*}\mathcal{L}_{i}\cong \mathcal{L}_{\beta(\sigma)(i)}$ for $1\leq i\leq m$. Now $\sigma\mapsto \beta(\sigma)^{-1}$ defines a group homomorphism $\mathrm{Gal}(\bar{k}/k)\rightarrow S_{m}$, and thus we have a finite Galois extension $l$ of $k$ with degree $\leq m!\leq n!$, such that any $\sigma\in\mathrm{Gal}(\overline{k}/l)$ permutes trivially the isomorphism classes of $\mathcal{L}_{1},...,\mathcal{L}_{m}$. We conclude that $\mathcal{L}_{1},...,\mathcal{L}_{m}$ belong to the set 
\begin{equation}\label{set}\{a\in \mathrm{Pic}_{X}^{0}(\overline{k})_{tor}| [\kappa(a):k]\leq n!\},\end{equation}
which is finite by assumption. Now we see that, for any essentially finite bundle $\mathcal{E}$ of rank $n$ with abelian monodromy, there are $n$ elements (allowing multiplicities) $\mathcal{L}_{1},...\mathcal{L}_{n}$ from the finite set \eqref{set} such that
$\mathcal{E}_{\overline{k}}$ splits into the direct sum of $\mathcal{L}_{1},...,\mathcal{L}_{n}$. If we have $\mathcal{E}$ and $\mathcal{E}'$ corresponding to the same $\mathcal{L}_{1},...,\mathcal{L}_{n}$, we have $\mathcal{E}_{\overline{k}}\cong\mathcal{E}'_{\overline{k}}$ and thus $\mathcal{E}\cong\mathcal{E}'$ by Lemma \ref{field_ext}. We conclude that the number of isomorphism classes of essentially finite bundles of rank $n$ with abelian monodromy is bounded by the number of ways to pick up $n$ elements (allowing multiplicities) from the finite set \eqref{set}.
\end{proof}
\begin{lem}[Jordan's theorem for finite group schemes]
\label{jordan_gpscheme}
Let $n>0$ be an integer. There exists a constant $J(n)$ such that for any field $k$ of characteristic $0$ and any finite subgroup scheme $\mathcal{G}$ of $\mathrm{GL}_{n,k}$, there is an abelian normal subgroup scheme of $\mathcal{G}$ of index $\leq J(n)$.
\end{lem}
\begin{proof}
We recall Jordan's theorem (\cite{jordan1878memoire}), saying that there exists a constant $J'(n)$, such that any finite abstract subgroup of $\mathrm{GL}_{n}(\overline{k})$ has an abelian normal subgroup of index $\leq J'(n)$. \par
We show that it suffices to take $J(n)$ to be $J'(n)^{2}$. We write $G$ for $\mathcal{G}(\bar{k})$, on which $\mathrm{Gal}(\bar{k}/k)$ acts. By Jordan's theorem, $G$ has an abelian normal subgroup $H'$ with $[G:H']\leq J'(n)$.  Now we evoke Chermak-Delgado theorem (Theorem 1.41 \cite{isaacs2008finite}), saying that there is an abelian characteristic subgroup $H$ of $G$, with $[G:H]\leq [G:H']^{2}\leq J(n)$. Since $H$ is characteristic, it is stabilized by the Galois action of $\mathrm{Gal}(\bar{k}/k)$ on $G$, and it descends to a finite subgroup scheme $\mathcal{H}$ of $\mathcal{G}$ of index $\leq J(n)$.
\end{proof}
\begin{thm}\label{main_thm}
Let $k$ be a field of characteristic $0$. The following assertions are equivalent: 
\begin{enumerate}
\item For any abelian variety $A$ over $k$ and any integer $n>0$, the set 
\[\{a\in A(\overline{k})_{tor}| [\kappa(a):k]\leq n\}\] is finite. 
\item For any pointed smooth projective variety $(X,x)$ over $k$ and any integer $n>0$, there are only finitely many isomorphism classes of essentially finite bundles of rank $n$ on $X$. 
\end{enumerate}
\end{thm}
\begin{proof}
$1\implies 2:$ by Lefschetz we can assume $X$ to be a curve. \par
Fix an integer $n$. Recall we have a constant $J(n)$ from Lemma \ref{jordan_gpscheme}. We claim that there is a pointed torsor $\pi:(Y,y)\rightarrow (X,x)$ under a finite étale group scheme from an another pointed  smooth proper curve $(Y,y)$, such that for any finite étale group scheme $\mathcal{G}$ of order $\leq J(n)$ together with a morphism $\pi_{1}^{N}(X,x)\rightarrow \mathcal{G}$, the composition $\pi_{1}^{N}(Y,y)\hookrightarrow\pi_{1}^{N}(X,x)\rightarrow \mathcal{G}$ is trivial. Indeed, consider 
\[\begin{tikzcd}&\pi_{1}^{et}(X_{\overline{k}},x_{\overline{k}})\ar[r,"\prod\rho"] &\prod_{\rho:\pi_{1}^{et}(X_{\overline{k}},x_{\overline{k}})\rightarrow G_{\rho}}G_{\rho}\end{tikzcd}\] where $\rho:\pi_{1}^{et}(X_{\overline{k}},x_{\overline{k}})\rightarrow G_{\rho}$ exhausts all possible group homomorphisms from $\pi_{1}^{et}(X_{\overline{k}},x_{\overline{k}})$ to a group of order $\leq J(n)$. There are only finitely many such homomorphisms since $\pi_{1}^{et}(X_{\overline{k}},x_{\overline{k}})$ is topologically finitely generated. The kernel of $\prod \rho$ is an open normal subgroup of $\pi_{1}^{et}(X_{\overline{k}},x_{\overline{k}})$. It is also characteristic by construction so it is stabilized by the Galois action of $\mathrm{Gal}(\overline{k}/k)$ on $\pi_{1}^{et}(X_{\overline{k}},x_{\overline{k}})$. Therefore the kernel descends to a pro-finite étale group scheme $\mathcal{K}$ over $k$ together with a homomorphism $\mathcal{K}\rightarrow\pi_{1}^{N}(X,x)$. Now we take $(Y,y)$ to be the pointed torsor defined by $\pi_{1}^{N}(X,x)\twoheadrightarrow \pi_{1}^{N}(X,x)/\mathcal{K}$. \par 
For any representation $\rho: \pi_{1}^{N}(X,x)\rightarrow\mathrm{GL}_{n}$, we write $\mathcal{E}_{\rho}$ for the essentially finite bundle associated to $\rho$ and $\mathcal{G}_{\rho}$ for the image of $\rho$. By Lemma \ref{jordan_gpscheme}, there exists an abelian normal subgroup scheme $\mathcal{H}_{\rho}$ of $\mathcal{G}_{\rho}$ of index $\leq J(n)$. By construction, the composition $\pi_{1}^{N}(Y,y)\hookrightarrow\pi_{1}^{N}(X,x)\twoheadrightarrow \mathcal{G}_{\rho}\twoheadrightarrow\mathcal{G}_{\rho}/\mathcal{H}_{\rho}$ is trivial. Therefore the composition $\pi_{1}^{N}(Y,y)\hookrightarrow\pi_{1}^{N}(X,x)\twoheadrightarrow \mathcal{G}_{\rho}$ factors through $\mathcal{H}_{\rho}$. Now the pullback $\pi^{*}\mathcal{E}_{\rho}$ corresponds to the representation of $\pi_{1}^{N}(Y,y)$ defined by the composition 
$\pi_{1}^{N}(Y,y)\hookrightarrow\pi_{1}^{N}(X,x)\twoheadrightarrow \mathcal{G}_{\rho}\hookrightarrow\mathrm{GL}_{n}$. This representation factors through $\mathcal{H}_{\rho}$ so it has abelian image. Applying Lemma \ref{key_lemma} to $(Y,y)$, for varying $\rho$, there are only finitely many isomorphism classes among $\pi^{*}\mathcal{E}_{\rho}$. Now note $\mathcal{E}_{\rho}$ is a direct summand of $\pi_{*}\pi^{*}\mathcal{E}_{\rho}$, so for varying $\rho$, there are only finitely many isomorphism classes among $\mathcal{E}_{\rho}$, by Krull–Schmidt.\par
$2\implies 1$: we prove by contradiction. There exists an abelian variety $A$, some integer $n>0$ and a sequence of torsion points $a_{i}$ of order exactly $N_{i}$ on $A$, $i\geq 0$, with $[\kappa(a_{i}):k]\leq n$ for any $i\geq 0$ and $N_{i}\rightarrow \infty$ when $i\rightarrow\infty$. Let $\widehat{A}$ be the dual abelian variety. Using the Weil pairing (which is perfect) 
\[A[N_{i}]\times\widehat{A}[N_{i}]\rightarrow \mu_{N_{i}}\]
and the point $a_{i}$, we get a faithfully flat morphism $\pi^{N}_{1}(\widehat{A}_{\kappa(a_{i})})\cong T\widehat{A}_{\kappa(a_{i})}\twoheadrightarrow \mu_{N_{i},\kappa(a_{i})}$, and thus a morphism
$\pi^{N}_{1}(\widehat{A})\cong T\widehat{A}\rightarrow \mathrm{Res}^{\kappa(a_{i})}_{k}\mu_{N_{i},\kappa(a_{i})}$ to the Weil restriction. The group scheme $\mu_{N_{i},\kappa(a_{i})}$ acts on $\mathbb{A}^{1}_{\kappa(a_{i})}$ and thus $ \mathrm{Res}^{\kappa(a_{i})}_{k}\mu_{N_{i},\kappa(a_{i})}$ acts on $\mathrm{Res}^{\kappa(a_{i})}_{k}\mathbb{A}^{1}_{\kappa(a_{i})} \cong \mathbb{A}^{[\kappa(a_{i}):k]}$. Therefore we get a representation of $\pi_{1}^{N}(\widehat{A})$ of dimension $[\kappa(a_{i}):k]$, and by taking direct sum with a trivial representation we get a representation of $\pi_{1}^{N}(\widehat{A})$ of dimension $n$ (as we have $[\kappa(a_{i}):k]\leq n$ by assumption). To summarize, for each $a_{i}$, $i\geq 0$, there is an associated $n$-dimensional representation of $\pi_{1}^{N}(\widehat{A})$, where $\pi^{N}_{1}(\widehat{A})_{\bar{k}}\twoheadrightarrow \mu_{N_{i},\bar{k}}$ appears as a direct summand after base change to $\overline{k}$. Therefore there are infinitely many isomorphism classes of representations of $\pi_{1}^{N}(\widehat{A})$ of dimension $n$, contradicting to the assumption.
\end{proof}
\begin{rmk}
The implication $2\implies 1$ is provided by the reviewer.
\end{rmk}
\begin{cor}\label{padic}
Let $k$ be a $p$-adic field. Then for any pointed smooth projective variety $(X,x)$ over $k$ and any integer $n>0$, there are only finitely many isomorphism classes of essentially finite bundles of rank $n$ on $X$. 
\end{cor}
\begin{proof}
It suffices to verify the first assertion in Theorem \ref{main_thm}. Indeed, there are only finitely many finite extensions of $k$ with given degree (Proposition 14 \uppercase\expandafter{\romannumeral2} §5 \cite{lang1994algebraic}), and Mattuck's theorem (\cite{mattuck1955abelian}) says that, torsion rational points on an abelian variety over a $p$-adic field are finite.
\end{proof}
A field is said to be sub-$p$-adic if it is a subfield of an another field finitely generated over a $p$-adic field. 
\begin{cor}\label{main_cor}
Let $k$ be a sub-$p$-adic field. Then for any pointed smooth projective variety $(X,x)$ over $k$ and any integer $n>0$, there are only finitely many isomorphism classes of essentially finite bundles of rank $n$ on $X$. 
\end{cor}
\begin{proof}
By Lemma \ref{field_ext} we can assume $k$ to be finitely generated over a $p$-adic field $k_{0}$. Apply induction on the  transcendence degree of $k$ over $k_{0}$. If the transcendence degree of $k$ over $k_{0}$ is $0$, $k$ is an another $p$-adic field and we are done by Corollary \ref{padic}. Now assume the transcendence degree is at least $1$. Let $A$ be an arbitrary abelian variety over $k$ and $n>0$ be any integer. We have to show that the set
\[\{a\in A(\overline{k})_{tor}| [\kappa(a):k]\leq n\}\] is finite. There is a smooth variety $S$ over $k_{0}$ with function field $k$, together with a spreading out $\mathcal{A}\rightarrow S$, which is an abelian scheme over $S$ with generic fiber $A$. Let $D$ be an irreducible Weil divisor on $S$ with generic point $\eta$. The local ring $\mathfrak{o}$ of $S$ at $\eta$ is a discrete valuation ring. Let $\widehat{\mathfrak{o}}$ denote the completion, $\widehat{k}$ denote the fraction field, and $\kappa$ denote the residue field. We have an inclusion 
\[\{a\in A(\overline{k})_{tor}| [\kappa(a):k]\leq n\}\subset \{a\in A_{\widehat{k}}(\overline{\widehat{k}})_{tor}| [\kappa(a):\widehat{k}]\leq n\}\]
and a bijection
\begin{equation}\label{etale_dvr} \begin{tikzcd}&\{a\in A_{\widehat{k}}(\overline{\widehat{k}})_{tor}| [\kappa(a):\widehat{k}]\leq n\}\ar[r,dash,"bijection"] &\{a\in \mathcal{A}_{\kappa}(\overline{\kappa})_{tor}| [\kappa(a):\kappa]\leq n\}\end{tikzcd}\end{equation}
because for any integer $N>0$ the group scheme $\mathcal{A}[N]_{\widehat{\mathfrak{o}}}$ is finite étale over $\widehat{\mathfrak{o}}$. By induction hypothesis the set 
\[\{a\in \mathcal{A}_{\kappa}(\overline{\kappa})_{tor}| [\kappa(a):\kappa]\leq n\}\] is finite.
\end{proof}
\begin{rmk}
There is an another interesting result in the same direction saying that essentially finite bundles are rather sparse: let $k$ be an arbitrary field of characteristic $0$ and $X$ be a geometrically connected smooth proper curve over $k$ of genus $\geq 2$. Then essentially finite bundles of rank $n\geq 2$ are not Zariski dense in the moduli space $M_{n,0}$ of semistable bundles of degree $0$ on $X$. See  \cite{ghiasabadi2023essentially}\cite{olsson2023moduli}. 
\end{rmk}
\section{Characteristic $p>0$ analogues}
We consider characteristic $p>0$ analogues of what we have established in last section. 
\begin{lem}\label{key_lemma_p}
Let $(X,x)$ be a pointed smooth proper curve over a field $k$ of characteristic $p>0$. Suppose that for any $n>0$, the set 
\[\{a\in \mathrm{Pic}_{X}^{0}(k^{s})_{tor}^{p'}| [\kappa(a):k]\leq n\}\]
is finite. Then for any $n>0$, there are only finitely many isomorphism classes of representations $\pi_{1}^{N}(X,x)^{p'}\rightarrow\mathrm{GL}_{n,k}$ with abelian monodromy.
\end{lem}
\begin{proof}
In case $k$ is a finite field, the assumption is trivially valid. The finiteness of representations follows from the fact that essentially finite bundles are semistable of degree $0$, and the moduli stack of semistable bundles of rank $n$ and of degree $0$ is of finite type over $k$.\par
In case $k$ is an infinite field, the lemma follows from the same argument as that for Lemma \ref{key_lemma}. Note that the category of representations of $\pi_{1}^{N}(X,x)^{p'}$ is semisimple. Also note that we have to make use of Lemma \ref{field_ext} in the proof and this is why we make the assumption that $k$ is infinite. 
\end{proof}
\begin{lem}
\label{jordan_gpscheme_p}
Let $n>0$ be an integer and $p$ be a prime. There exists a constant $J(n,p)$ depending on $n$ and $p$ such that, for any field $k$ of characteristic $p>0$ and any finite étale subgroup scheme $\mathcal{G}$ of $\mathrm{GL}_{n,k}$  of order coprime to $p$, $\mathcal{G}$ has an abelian normal subgroup scheme of index $\leq J(n,p)$.
\end{lem}
\begin{proof}
We use Theorem 0.4 \cite{larsen2011finite} as a characteristic $p>0$ alternative of Jordan's theorem. The argument is identical to that for Lemma \ref{jordan_gpscheme}.  
\end{proof}
\begin{prop}\label{main_prop_p}
Let $k$ be a field of characteristic $p>0$. The following assertions are equivalent:
\begin{enumerate}
\item For any abelian variety $A$ over $k$ and any integer $n>0$, the set 
\[\{a\in A(k^{s})^{p'}_{tor}| [\kappa(a):k]\leq n\}\] is finite. 
\item For any pointed smooth projective variety $(X,x)$ over $k$ and any integer $n>0$, there are only finitely many isomorphism classes of representations $\pi_{1}^{N}(X,x)^{p'}\rightarrow\mathrm{GL}_{n}$.
\end{enumerate}
\end{prop}
\begin{proof}
We sketch the proof of the implication $1\implies 2$, as there are some necessary modifications from the argument for Theorem \ref{main_thm}. By Lefschetz we assume $X$ is a curve. Fix an integer $n>0$. Recall we have a constant $J(n,p)$ from Lemma \ref{jordan_gpscheme_p}. There is a pointed torsor $\pi:(Y,y)\rightarrow (X,x)$ under a finite étale group scheme of order coprime to $p$ from an another pointed  smooth proper curve $(Y,y)$, such that for any finite étale group scheme $\mathcal{G}$ of order $\leq J(n,p)$ and of order coprime to $p$ together with a morphism $\pi_{1}^{N}(X,x)^{p'}\rightarrow \mathcal{G}$, the composition $\pi_{1}^{N}(Y,y)^{p'}\hookrightarrow\pi_{1}^{N}(X,x)^{p'}\rightarrow \mathcal{G}$ is trivial. For any representation $\rho: \pi_{1}^{N}(X,x)^{p'}\rightarrow \mathrm{GL}_{n}$ we have the corresponding essentially finite bundle $\mathcal{E}_{\rho}$. By Lemma \ref{key_lemma_p} there are only finitely many isomorphism classes of $\pi^{*}\mathcal{E}_{\rho}$ among all $\rho$. We use the fact that $\pi$ has degree coprime to $p$ to make sure that the monodromy group of $\pi_{*}\pi^{*}\mathcal{E}_{\rho}$ has order coprime to $p$, and thus $\pi_{*}\pi^{*}\mathcal{E}_{\rho}$, which corresponds a priori to a representation of $\pi_{1}^{N}(X,x)$, actually corresponds to a representation of $\pi_{1}^{N}(X,x)^{p'}$.  The category of representations of $\pi_{1}^{N}(X,x)^{p'}$ is semisimple, so we conclude by looking at rank $n$ direct summands of  $\pi_{*}\pi^{*}\mathcal{E}_{\rho}$ among all possible $\rho$.\par
The implication $2\implies 1$ can be proved using the same argument as that for Theorem \ref{main_thm}.
\end{proof}
\begin{cor}\label{main_cor_p}
Let $k$ be finitely generated over a finite field (e.g. a global function field). Then for any pointed smooth projective variety $(X,x)$ over $k$ and any integer $n>0$, there are only finitely many isomorphism classes of representations $\pi_{1}^{N}(X,x)^{p'}\rightarrow \mathrm{GL}_{n}$. 
\end{cor}
\begin{proof}
As in the proof of Corollary \ref{main_cor}, we use a spreading out argument to reduce to the case that $k$ is a finite field. We conclude by Proposition \ref{main_prop_p}, where the first assertion is trivially valid over a finite field. The only thing to be noted is that, for a bijection as that of \eqref{etale_dvr} , we have to work with finite étale group schemes over a complete discrete valuation ring, and this is one another place where we have to avoid the prime $p$. 
\end{proof}

 \bibliographystyle{plain}
 \bibliography{Finiteness_FiniteBundle}

\begin{thebibliography}{10}

\bibitem{antei2010comparison}
Marco Antei.
\newblock Comparison between the fundamental group scheme of a relative scheme
  and that of its generic fiber.
\newblock {\em Journal de th{\'e}orie des nombres de Bordeaux}, 22(3):537--555,
  2010.

\bibitem{10.1215/S0012-7094-03-11723-8}
Carlo Gasbarri.
\newblock {Heights of vector bundles and the fundamental group scheme of a
  curve}.
\newblock {\em Duke Mathematical Journal}, 117(2):287 -- 311, 2003.

\bibitem{ghiasabadi2023essentially}
Archia Ghiasabadi and Stefan Reppen.
\newblock Essentially finite {$G$}-torsors.
\newblock {\em Bulletin des Sciences Math{\'e}matiques}, 188:103334, 2023.

\bibitem{grothendieck2003revetements}
Alexander Grothendieck and Michele Raynaud.
\newblock {\em Rev{\^e}tements {\'e}tales et groupe fondamental: S{\'e}minaire
  de G{\'e}ometrie Alg{\'e}brique du Bois Marie 1960-61, SGA1}.
\newblock Springer, 2003.

\bibitem{isaacs2008finite}
I~Martin Isaacs.
\newblock {\em Finite group theory}, volume~92.
\newblock American Mathematical Soc., 2008.

\bibitem{jordan1878memoire}
Camille~M Jordan.
\newblock M{\'e}moire sur les {\'e}quations diff{\'e}rentielles lin{\'e}aires
  {\`a} int{\'e}grale alg{\'e}brique.
\newblock {\em Journal f{\"u}r die reine und angewandte Mathematik (Crelles
  Journal)}, 1878(84):89--215, 1878.

\bibitem{lang1994algebraic}
Serge Lang.
\newblock {\em Algebraic number theory}, volume 110.
\newblock Springer Science \& Business Media, 1994.

\bibitem{larsen2011finite}
Michael Larsen and Richard Pink.
\newblock Finite subgroups of algebraic groups.
\newblock {\em Journal of the American Mathematical Society}, 24(4):1105--1158,
  2011.

\bibitem{mattuck1955abelian}
Arthur Mattuck.
\newblock Abelian varieties over p-adic ground fields.
\newblock {\em Annals of Mathematics}, 62(1):92--119, 1955.

\bibitem{nori1982fundamental}
Madhav~V Nori.
\newblock The fundamental group-scheme.
\newblock {\em Proceedings Mathematical Sciences}, 91(2):73--122, 1982.

\bibitem{olsson2023moduli}
Ludvig Olsson, Stefan Reppen, and Tuomas Tajakka.
\newblock Moduli of $\mathcal{G}$-bundles under nonconnected group schemes and
  nondensity of essentially finite bundles.
\newblock {\em arXiv preprint arXiv:2311.05326}, 2023.

\bibitem{galfun}
Tamás Szamuely.
\newblock {\em Galois groups and fundamental groups}, volume 117.
\newblock Cambridge University Press, 2009.

\end{thebibliography}
\end{document}